\documentclass[oneside, a4paper]{article}
\usepackage[hidelinks]{hyperref}
\usepackage{amsmath, amsthm, amssymb, amsfonts}
\usepackage[left=2.5cm,right=2.5cm,top=3cm,bottom=3cm,a4paper]{geometry}
\usepackage{thmtools}
\newtheorem{theorem}{Theorem}[section]
\newtheorem{corollary}[theorem]{Corollary}
\newtheorem{lemma}[theorem]{Lemma}

\usepackage{array}
\makeatletter
\newcommand{\thickhline}{%
    \noalign {\ifnum 0=`}\fi \hrule height 1pt
    \futurelet \reserved@a \@xhline
}
\newcolumntype{"}{@{\hskip\tabcolsep\vrule width 1pt\hskip\tabcolsep}}
\makeatother

\title{\Large{Integral matrices as diagonal quadratic forms}}
\author{\large{Jungin Lee}}
\date{}

\usepackage{fancyhdr}

\newcommand\shorttitle{Integral matrices as diagonal quadratic forms}
\newcommand\authors{Jungin Lee}

\begin{document}
\maketitle

\vspace{0mm}

\textbf{Abstract.} In this paper, we investigate the conditions under which a diagonal quadratic form $\sum_{i=1}^{m}a_i X_i^2$ represents every $n \times n$ integral matrix, where $a_i$ ($1 \leq i \leq m$) are integers. For $n=2$, we give a necessary and sufficient condition. Also we give some sufficient conditions for each $n \geq 2$ where $a_i$ ($1 \leq i \leq m$) are pairwisely coprime. 

\vspace{5mm}


\vspace{5mm}

\section{Introduction}

There are some papers on Waring's problem for integral matrices. Vaserstein \cite{var86} proved that for $n \geq 2$, every $n \times n$ integral matrix is a sum of three squares. Richman \cite{ric87} proved that for $n \geq k \geq 2$, every $n \times n$ integral matrix is a sum of seven $k$-th powers. In this paper, we generalize the result of \cite{var86} to the the diagonal quadratic forms. \\
The paper is organized as follows. In section \ref{section2} we provide a necessary and sufficient condition that an integral quadratic form $\sum_{i=1}^{m}a_i X_i^2$ is universal over $M_2(\mathbb{Z})$. For a positive integer $n \geq 2$, let $f(n)$ denote the smallest positive integer $m$ such that for every pairwise coprime $a_1, \cdots, a_m \in \mathbb{Z}$, $\sum_{i=1}^{m}a_iX_i^2$ is universal over $M_n(\mathbb{Z})$. In section \ref{section3} we give upper bounds of $f(n)$ for each $n \geq 2$. The result of section \ref{section3} is given in the table \ref{table1} below. \\

\begin{table}[h]
\centering
\begin{tabular}{|c|c|}\hline
  n & f(n) \\ \hline
  2 & 4 \\ \hline
  3 & $\leq $6 \\ \hline
  $2k \,\,(k \geq 2)$ & $\leq $6 \\ \hline
  $2k+1 \,\,(k \geq 2)$ & $\leq $8 \\ \hline
\end{tabular}
\caption{Upper bounds of $f(n)$}
\label{table1} 
\end{table}

\section{The Case $n=2$} \label{section2}

\begin{theorem} \label{thm2.1}
An integral quadratic form $\sum_{i=1}^{m}a_i X_i^2$ is universal over $M_2(\mathbb{Z})$ if and only if there is no prime which divides $m-1$ numbers of $a_1, \cdots, a_m$ and there exist three numbers of $a_1, \cdots, a_m$ which are not multiples of $4$.
\end{theorem}
\begin{proof}
If $\sum_{i=1}^{m}a_i X_i^2$ is universal over $M_2(\mathbb{Z})$, 
then it is universal over $M_2(\mathbb{Z}_r)$ for every positive integer $r$. 
Suppose that a prime $p$ divides $m-1$ numbers of $a_1, \cdots, a_m$. Then $\sum_{i=1}^{m}a_i X_i^2$ is universal over $M_2(\mathbb{Z}_p)$ if and only if $X^2$ is universal over $M_2(\mathbb{Z}_p)$, which is impossible since 
$ \begin{bmatrix}
0 & 0\\ 
0 & 0
\end{bmatrix} ^2
=  \begin{bmatrix}
0 & 0\\ 
1 & 0
\end{bmatrix}^2$ in $M_2(\mathbb{Z}_p)$. Now suppose that $m-2$ numbers of $a_1, \cdots a_m$ (without loss of generality, namely, $a_3, \cdots, a_m$) are multiples of $4$. If $\sum_{i=1}^{m}a_i X_i^2$ is universal over $M_2(\mathbb{Z}_4)$, then $a_1X_1^2+a_2X_2^2$ is universal over $M_2(\mathbb{Z}_4)$, so $a_1$ and $a_2$ are odd by above result. However, $X^2+Y^2$ and $X^2-Y^2$ are not universal over $M_2(\mathbb{Z}_4)$, since each of them does not represent $\begin{bmatrix}
1 & 0\\ 
0 & 3
\end{bmatrix}$ and $\begin{bmatrix}
0 & 0\\ 
0 & 2
\end{bmatrix}$, respectively. Thus the condition in the theorem is necessary. \\
Suppose that the condition in the theorem holds. It is easy to show that
$\sum_{i=1}^{m-1}a_i\begin{bmatrix}
x_i & y_i\\ 
z_i & c_i-x_i
\end{bmatrix}^2+a_m\begin{bmatrix}
0 & N\\ 
1 & 0
\end{bmatrix}^2=\begin{bmatrix}
p & q\\ 
r & s
\end{bmatrix}$ if and only if the four equations below hold. 
\begin{subequations} \label{eq1}
\begin{gather}
\sum_{i=1}^{m-1}a_i(x_i^2+y_iz_i)+a_mN=p \label{eq1a} \\
\sum_{i=1}^{m-1}a_ic_iy_i=q \label{eq1b} \\
\sum_{i=1}^{m-1}a_ic_iz_i=r \label{eq1c} \\
\sum_{i=1}^{m-1}a_ic_ix_i=\frac{p-s+\sum_{i=1}^{m-1}a_ic_i^2}{2} \label{eq1d}
\end{gather}
\end{subequations} \\

\textbf{Case I.} There exist three odd numbers among $a_1, \cdots, a_m$. Without loss of generality, assume that $a_1, a_2$ and $a_m$ are odd. Let $c_i=2a_i \,\,(3 \leq i \leq m-1)$, $c_1=a_m$ and $c_2 \in \left \{ a_2, 2a_2 \right \} \,\,(c_2 \equiv p-s+1 \,\,(mod \,\,2))$. For a prime factorization $a_m=p_1^{e_1} \cdots p_r^{e_r}$, there exist $\alpha_j \neq \beta_j$ such that for every $1 \leq j \leq r$, $p_j$ does not divide $a_{\alpha_j}$ and $a_{\beta_j}$. If we replace $c_{\alpha_j}$ to $p_j c_{\alpha_j}$ for $1 \leq j \leq r$, $p_j$ does not divide $a_{\alpha_j}c_{\alpha_j}^2+a_{\beta_j}c_{\beta_j}^2$, $gcd(a_1c_1, \cdots, a_{m-1}c_{m-1})=1$ and $p-s+\sum_{i=1}^{m-1}a_ic_i^2 \equiv 0 \,\,(mod \,\,2)$. Now we can choose $x_i, y_i, z_i$ ($1 \leq i \leq m-1$) which give a solution to the equations \eqref{eq1b}, \eqref{eq1c} and \eqref{eq1d}. \\ 
Denote $A=\sum_{i=1}^{m-1}a_i(x_i^2+y_iz_i)-p$ and $u_i=\frac{a_m}{p_i^{e_i}}$ for each $1 \leq i \leq r$. 
If we replace $y_{\alpha_i}, y_{\beta_i}, z_{\alpha_i}$ and $z_{\beta_i}$ by 
$y_{\alpha_i}+k u_i a_{\beta_i}c_{\beta_i}, \,
y_{\beta_i}-k u_i a_{\alpha_i}c_{\alpha_i}, \,
z_{\alpha_i}+t u_i a_{\beta_i}c_{\beta_i}$ and 
$z_{\beta_i}-t u_i a_{\alpha_i}c_{\alpha_i}$, the equations \eqref{eq1b}, \eqref{eq1c} and \eqref{eq1d} still hold and $A$ is changed by 
$A+u_ia_{\alpha_i}a_{\beta_i}(tP_i+k(Q_i+tR_i))$, where 
$P_i=c_{\beta_i}y_{\alpha_i}-c_{\alpha_i}y_{\beta_i}, \,
Q_i=c_{\beta_i}z_{\alpha_i}-c_{\alpha_i}z_{\beta_i}$
and $R_i=u_i(a_{\alpha_i}c_{\alpha_i}^2+a_{\beta_i}c_{\beta_i}^2)$. 
Since $R_i$ is not a multiple of $p_i$, there exists $t \in \mathbb{Z}$ such that $Q_i+tR_i$ is not a multiple of $p_i$. Now there exists $k \in \mathbb{Z}$ such that $p_i^{e_i}$ divides $A+u_ia_{\alpha_i}a_{\beta_i}(tP_i+k(Q_i+tR_i))$. 
After repeating this process for $1 \leq i \leq r$, $A$ becomes a multiple of $a_m$, so there exists an integer $N$ such that the equation \eqref{eq1a} holds. \\

\textbf{Case II.} There exist exactly two odd numbers among $a_1, \cdots, a_m$, one number is of the form $4k+2$ and not all of $p-s$, $q$ and $r$ are even. Without loss of generality, assume that $a_1$ and $a_2$ are odd and $a_m=2a_m'$ for an odd number $a_m'$. 
Let $c_i=2a_i \,\,(3 \leq i \leq m-1)$, $c_1=a_m'$ and $c_2 \in \left \{ a_2, 2a_2 \right \} \,\,(c_2 \equiv p-s+1 \,\,(mod \,\,2))$. 
By repeating the argument of Case I, we can obtain $c_i, x_i, y_i$ and $z_i$ such that the equations \eqref{eq1b}, \eqref{eq1c} and \eqref{eq1d} hold and $A$ is a multiple of $a_m'$. Note that this procedure does not change the parity of $c_i$, so $c_1$ is odd and $c_2 \equiv p-s+1 \,\,(mod \,\,2)$ after the procedure. Since $A \equiv x_1+y_1z_1+x_2+y_2z_2-p \,\, (mod \,\,2)$, it is enough to show that we can change the parity of $x_1+y_1z_1+x_2+y_2z_2$. \\

\begin{enumerate}
  \item $p-s$ is odd : $c_2$ is even. If we replace $x_1$ and $x_2$ by $x_1+a_2c_2$ and $x_2-a_1c_1$, the equation \eqref{eq1d} still holds and the parity of $x_1+x_2$ changes. 
  \item $p-s$ is even, $q$ is odd : $c_2$ is odd. By the equation \eqref{eq1b}, $y_1+y_2$ is odd. If we replace $z_1$ and $z_2$ by $z_1+a_2c_2$ and $z_2-a_1c_1$, the equation \eqref{eq1c} still holds and the parity of $y_1z_1+y_2z_2$ changes
  . 
  \item $p-s$ is even, $r$ is odd : symmetric to (2). \\
\end{enumerate}

\textbf{Case III.} There exist exactly two odd numbers among $a_1, \cdots, a_m$, one number is of the form $4k+2$ and all of $p-s$, $q$ and $r$ are even. Without loss of generality, assume that $a_2$ and $a_m$ are odd and $a_1=2a_1'$ for an odd number $a_1'$. 
Let $c_i=2a_i \,\,(2 \leq i \leq m-1)$ and $c_1 \in \left \{ a_m, 2a_m \right \} \,\,(c_1 \equiv \frac{p-s}{2} \,\,(mod \,\,2))$. Then $p-s+\sum_{i=1}^{m-1}a_ic_i^2 \equiv p-s+2c_1 \equiv 0 \,\, (mod \,\, 4)$. Now $q$ and $r$ are even and $gcd(a_1c_1, \cdots, a_{m-1}c_{m-1})=2$, so we can choose $x_i, y_i, z_i$ ($1 \leq i \leq m-1$) which give a solution to the equations \eqref{eq1b}, \eqref{eq1c} and \eqref{eq1d}. Now by repeating the argument of Case I, we can obtain a solution of \eqref{eq1a}, \eqref{eq1b}, \eqref{eq1c} and \eqref{eq1d}. 
\end{proof} 

For the case $m=3$, we can restate theorem \ref{thm2.1} as follows. 

\begin{corollary} \label{cor2.2}
An integral quadratic form $aX^2+bY^2+cZ^2$ is universal over $M_2(\mathbb{Z})$ if and only if $(a,b)=(b,c)=(c,a)=1$ and $abc$ is not a multiple of $4$. 
\end{corollary}

\section{General Case} \label{section3}
By theorem \ref{thm2.1}, $f(2)=4$. In this section, we will assume that $a_1, \cdots, a_m$ are pairwise coprime integers. First we prove the following lemma based on lemma 3 of \cite{var87}.

\begin{lemma} \label{lem3.1}
For every integer $p,q \geq 2$, $f(p+q) \leq 2+max\left \{ f(p),f(q) \right \}$.
\end{lemma}
\begin{proof}
Let $t_1$ and $t_2$ be integers such that $a_1t_1+a_2t_2=1$ and $m=max\left \{ f(p),f(q) \right \}$. Then for every $A=\begin{bmatrix}
X & Y\\ 
Z & W
\end{bmatrix} \in M_{p+q}(\mathbb{Z})$ ($X\in M_{p}(\mathbb{Z})$ and $W\in M_{q}(\mathbb{Z})$), $A-a_1\begin{bmatrix}
O & t_1Y\\ 
t_1Z & I
\end{bmatrix}^2
-a_2\begin{bmatrix}
O & t_2Y\\ 
t_2Z & I
\end{bmatrix}^2
=\begin{bmatrix}
X' & O\\ 
O & W'
\end{bmatrix}$. Now there exist $X_i \in M_p(\mathbb{Z}), \, W_i \in M_q(\mathbb{Z})\: \: (3\leq i\leq m+2)$ such that $X'=\sum_{i=3}^{m+2}a_iX_i^2$ and $W'=\sum_{i=3}^{m+2}a_iW_i^2$, so we can write $\begin{bmatrix}
X' & O\\ 
O & W'
\end{bmatrix}=\sum_{i=3}^{m+2}a_i\begin{bmatrix}
X_i & O\\ 
O & W_i
\end{bmatrix}^2$. Thus $f(p+q) \leq m+2$. 
\end{proof}

\begin{lemma} \label{lem3.2}
$f(3) \leq 6$. 
\end{lemma}
\begin{proof}
Without loss of generality, assume that $a_2, a_3$ and $a_4$ are odd. Fix a matrix $A \in M_3(\mathbb{Z})$.   Let $t_1, t_2, \cdots, t_6$ be integers such that $a_1t_1+a_2t_2=a_3t_3+a_4t_4=a_5t_5+a_6t_6=1$ and $m=A_{31}-1$. Choose $u \in \left \{ 0,1 \right \}$ such that $u \equiv tr(A)-a_1-1 \,\,(mod \,\,2)$. Then 
$A'=A-a_1\begin{bmatrix}
1 & 0 & 0\\ 
0 & 0 & 0\\ 
mt_1 & 0 & 0
\end{bmatrix}^2
-a_2\begin{bmatrix}
1 & 0 & 0\\ 
0 & u & 0\\ 
mt_2 & 0 & 0
\end{bmatrix}^2
=\begin{bmatrix}
c_1 & * & *\\ 
c_2 & * & *\\ 
1 & c_3 & c_4
\end{bmatrix}$ and $tr(A')$ is even. Let $\tilde{c_3}=c_3+a_3+a_4$ and $T=\begin{bmatrix}
1 & -\tilde{c_3} & c_1\\ 
0 & 1 & c_2\\ 
0 & 0 & 1
\end{bmatrix}$. Then $T^{-1}=\begin{bmatrix}
1 & \tilde{c_3} & -c_2\tilde{c_3}-c_1\\ 
0 & 1 & -c_2\\ 
0 & 0 & 1
\end{bmatrix}$ and $T^{-1}A'T=\begin{bmatrix}
0 & a & b\\ 
0 & c & d\\ 
1 & -a_3-a_4 & e
\end{bmatrix}$. $tr(T^{-1}A'T)=tr(A')$ implies that $c-e$ is even. 
Let $Q_3=t_3b, Q_4=t_4b, u_3=t_3\frac{c-e-(a_3+a_4)}{2}, u_4=t_4\frac{c-e-(a_3+a_4)}{2}, P=d-a_3(1+u_3+u_3^2)-a_4(1+u_4+u_4^2), P_3=t_3(a-(a_5+a_6)P), P_4=t_4(a-(a_5+a_6)P)$ and $Q=c-(a_3u_3+a_4u_4)=e+a_3(1+u_3)+a_4(1+u_4)$. Also let 
$X_i=\begin{bmatrix}
0 & P_i & Q_i\\ 
0 & u_i & 1+u_i+u_i^2\\ 
0 & -1 & -1-u_i
\end{bmatrix} \,\, (3 \leq i \leq 4)$ and 
$Y_i=\begin{bmatrix}
0 & 0 & P\\ 
t_i & 0 & t_iQ\\
0 & 1 & 0
\end{bmatrix} \,\, (5 \leq i \leq 6)$. (The structure of $X_i$ comes from the proof of lemma 2 in \cite{var86}.) Then $X_i=X_i^4 \,\, (3 \leq i \leq 4)$ and 
$T^{-1}A'T=a_3(X_3^2)^2+a_4(X_4^2)^2+a_5Y_5^2+a_6Y_6^2$. If we denote $\widetilde{X_i}=TX_iT^{-1} \,\, (3 \leq i \leq 4)$ and $\widetilde{Y_i}=TY_iT^{-1} \,\, (5 \leq i \leq 6)$, then $A'=a_3(\widetilde{X_3}^2)^2+a_4(\widetilde{X_4}^2)^2+a_5\widetilde{Y_5}^2+a_6\widetilde{Y_6}^2$. 
\end{proof} 

\begin{lemma} \label{lem3.3}
Let $a_1$ and $a_2$ be relatively prime odd integers. Then for every integer $m$, there exist $c,d,u,v \in \mathbb{Z}$ such that $gcd(a_1(u-c), a_2(v-d))=1$ and $a_1(c^2+2u)+a_2(d^2+2v)=m$. 
\end{lemma}
\begin{proof}
Let the prime factors of $a_1$ be $p_1, \cdots, p_r$. For each $p_i$, there exists $d_i \in \mathbb{Z}$ such that $p_i$ does not divide $m-a_2(d_i^2+2d_i)$. Chinese remainder theorem implies that there exists $d \in \mathbb{Z}$ such that $gcd(a_1, m-a_2(d^2+2d))=1$. By the same reason, there exists $c \in \mathbb{Z}$ such that $gcd(a_2, m-a_1(c^2+2c))=1$. Thus we can find $c, d \in \mathbb{Z}$ such that $gcd(a_1a_2, m-a_1(c^2+2c)-a_2(d^2+2d))=1$. It is still true if we replace $c$ by $c+a_2$, so we can assume that $m-(c+d)$ is even. \\
Let $A=a_1(u-c)+a_2(v-d)$ and $B=a_1(c^2+2c)+a_2(d^2+2d)$. Then $a_1(c^2+2u)+a_2(d^2+2v)=2A+B$. By above argument, there exist $c,d \in \mathbb{Z}$ such that $m-B$ is even and $gcd(a_1a_2, m-B)=1$. 
Since $gcd(a_1,a_2)=gcd(m-B,a_2)=1$, we can choose $u,v,t \in \mathbb{Z}$ such that $A=\frac{m-B}{2}$ and $gcd(u+ta_2-c, A)=1$. 
If we replace $u$ and $v$ by $u+ta_2$ and $v-ta_1$, $A=\frac{m-B}{2}$ still holds and $gcd(a_1(u-c), a_2(v-d))=gcd(a_1(u-c), A)=1$. 
\end{proof} 

Now we provide an upper bound of $f(n)$ for each $n \geq 2$. 

\begin{theorem} \label{thm3.4}
For every positive integer $n$, $f(2n) \leq 6$. 
\end{theorem}
\begin{proof}
Since $f(2)=4$, we can assume that $n \geq 2$. Without loss of generality, assume that $a_1$ and $a_2$ are odd. Fix a matrix
$\begin{bmatrix}
P & Q\\ 
R & S
\end{bmatrix} \in M_{2n}(\mathbb{Z})$ ($P,Q,R,S \in M_{n}(\mathbb{Z})$). Let $t_1$ and $t_2$ be integers such that $a_1t_1+a_2t_2=1$, $Q=\begin{bmatrix}
q_1, \cdots, q_n
\end{bmatrix}$ and
$R=\begin{bmatrix}
r_1, \cdots, r_n
\end{bmatrix}^T$. Also let
\begin{equation}
X_1=\begin{bmatrix}
c & u &  & & \\ 
1 & 1 &  & & \\ 
 &  & 1 & & \\ 
 &  & & \ddots & \\ 
 &  &  & & 1
\end{bmatrix}, \, 
X_2=\begin{bmatrix}
d & v &  & & \\ 
1 & 1 &  & & \\ 
 &  & 1 & & \\ 
 &  & & \ddots & \\ 
 &  &  & & 1
\end{bmatrix}
\end{equation}
and $
B_1=\begin{bmatrix}
\alpha_1, \alpha_2, t_1q_3, \cdots, t_1q_n
\end{bmatrix}, \, 
B_2=\begin{bmatrix}
\alpha_3, \alpha_4, t_2q_3, \cdots, t_2q_n
\end{bmatrix}$, \\
$C_1=\begin{bmatrix}
\beta_1-\beta_2, \beta_2, t_1r_3, \cdots, t_1r_n
\end{bmatrix}^T, \, 
C_2=\begin{bmatrix}
\beta_3-\beta_4, \beta_4, t_2r_3, \cdots, t_2r_n
\end{bmatrix}^T$
where $c,d,u,v \in \mathbb{Z}$ and $\alpha_i, \beta_i \in M_{n \times 1}(\mathbb{Z})$ for $1 \leq i \leq 4$. 
Now it is easy to show that $a_1B_1X_1+a_2B_2X_2=Q$ if and only if the equations \eqref{eq2a} and \eqref{eq2b} hold, and $a_1X_1C_1+a_2X_2C_2=R$ if and only if the equations \eqref{eq2c} and \eqref{eq2d} hold. 

\begin{subequations} \label{eq6}
\begin{gather}
a_1(u-c)\alpha_1+a_2(v-d)\alpha_3=q_2-q_1 \label{eq2a} \\
a_1\alpha_2+a_2\alpha_4=q_1-(a_1c\alpha_1+a_2d\alpha_3) \label{eq2b} \\
a_1\beta_1+a_2\beta_3=r_2 \label{eq2c} \\
a_1(u-c)\beta_2+a_2(v-d)\beta_4=r_1-(a_1c\beta_1+a_2d\beta_3) \label{eq2d}
\end{gather}
\end{subequations}

By lemma \ref{lem3.3}, there exist $c,d,u,v \in \mathbb{Z}$ such that $gcd(a_1(u-c), a_2(v-d))=1$ and $a_1(c^2+2u)+a_2(d^2+2v)=tr(S)-tr(P)-(n-1)(a_1+a_2)$. Thus we can find $\alpha_1$, $\alpha_3$, $\beta_1$ and $\beta_3$ which satisfy the equations \eqref{eq2a} and \eqref{eq2c},
and $\alpha_2$, $\alpha_4$, $\beta_2$ and $\beta_4$ which satisfy the equations \eqref{eq2b} and \eqref{eq2d}. Then
\begin{equation} \label{eq7}
\begin{bmatrix}
P & Q\\ 
R & S
\end{bmatrix}
-a_1\begin{bmatrix}
O & B_1\\ 
C_1 & X_1
\end{bmatrix}^2
-a_2\begin{bmatrix}
O & B_2\\ 
C_2 & X_2
\end{bmatrix}^2=
\begin{bmatrix}
P_0 & O\\ 
O & S_0
\end{bmatrix}
\end{equation}
where $P_0=P-a_1B_1C_1-a_2B_2C_2$ and $S_0=S-a_1(C_1B_1+X_1^2)-a_2(C_2B_2+X_2^2)$. By simple calculation, $tr(a_1X_1^2+a_2X_2^2)=(n-1)(a_1+a_2)+(a_1(c^2+2u)+a_2(d^2+2v))=tr(S)-tr(P)$ and $tr(P_0-S_0)=tr(P)-tr(S)-a_1tr(B_1C_1-C_1B_1)-a_2tr(B_2C_2-C_2B_2)+tr(a_1X_1^2+a_2X_2^2)=0$.  \\
From the result in \cite{lef94}, we can deduce that there exist $X,Y \in M_n(\mathbb{Z})$ such that $P_0-S_0=XY-YX$. Let $t_3, t_4, t_5$ and $t_6$ be integers such that $a_3t_3+a_4t_4=a_5t_5+a_6t_6=1$. Then
\begin{equation}
\begin{bmatrix}
P_0 & O\\ 
O & S_0
\end{bmatrix} 
=
a_3\begin{bmatrix}
O& t_3X\\ 
Y & O
\end{bmatrix} ^2
+
a_4\begin{bmatrix}
O & t_4X\\ 
Y & O
\end{bmatrix} ^2
+
a_5\begin{bmatrix}
O & t_5N\\ 
I & O
\end{bmatrix} ^2
+
a_6\begin{bmatrix}
O & t_6N\\ 
I & O
\end{bmatrix} ^2
\end{equation}
where $N=P_0-XY=S_0-YX$. Thus $f(2n) \leq 6$. 
\end{proof} 

\begin{theorem} \label{thm3.5}
For every positive integer $n$, $f(2n+1) \leq 8$. 
\end{theorem}
\begin{proof}
If $n=1$, $f(3) \leq 6$ by lemma \ref{lem3.2}. For $n \geq 2$, lemma \ref{lem3.1} and theorem \ref{thm3.4} implies $f(2n+1) \leq 2+max\left \{ f(3),f(2n-2) \right \} \leq 8$.
\end{proof}

\vspace{0.5cm}

Department of Mathematical Sciences, Seoul National University, Seoul 151-747, Korea 

e-mail: moleculesum@snu.ac.kr

\end{document}